\documentclass[11pt,reqno]{amsart}
\usepackage{amscd,amssymb,amsmath,amsthm}
\usepackage[arrow,matrix]{xy}
\usepackage{graphicx}
\usepackage{cite}

\newtheorem{thm}[subsection]{Theorem}
\newtheorem{lemma}[subsection]{Lemma}
\newtheorem{pro}[subsection]{Proposition}

\newtheorem{rk}[subsection]{Remark}

\numberwithin{equation}{section} 

\newcommand{\bea}{\begin{eqnarray}}
\newcommand{\eea}{\end{eqnarray}}

\begin{document}
\title[Dynamics of linear maps of idempotent measures]{Dynamics of linear maps of idempotent measures}

\author{U.A. Rozikov, \, M.M. Karimov}

 \address{U.\ A.\ Rozikov\\ Institute of mathematics and information technologies,
Tashkent, Uzbekistan.}
\email {rozikovu@yandex.ru}

\address{M.\ M.\ Karimov\\ Namangan State University, Namangan, Uzbekistan.}
\email {mmkarimov86@rambler.ru}

\begin{abstract} We describe all linear operators which maps $n-1$-dimensional simplex of idempotent measures to itself.
Such operators divided to two classes: the first class contains all $n\times n$-matrices with non-negative entries which has at least one zero-row; the second class contains all $n\times n$-matrices with non-negative entries which in each row and in each column has exactly one non-zero entry. These matrices play a role of the stochastic matrices in case of idempotent measures. For both classes of linear maps we find fixed points. We also study the dynamical systems generated by the linear maps of the set of idempotent measures.

{\it AMS classifications (2010):} 37C25; 47Axx.\\[2mm]

{\it Keywords:} Linear map; dynamical system; idempotent measure; fixed point; trajectory; limit point.
\end{abstract}
 \maketitle
\section{Introduction} \label{sec:intro}

 In the most general sense, a {\it dynamical system} is a tuple $(T, M,\Phi)$ where $T$ is a monoid,
written additively, $M$ is a set and $\Phi$ is a function
$\Phi: U \subset T \times M \to M$ with
$$ I(x) = \{ t \in T : (t,x) \in U \}, \,  \Phi(0,x) = x, $$ $$  \Phi(t_2,\Phi(t_1,x)) = \Phi(t_1 + t_2, x),
\, \mbox{for} \, t_1, t_2, t_1 + t_2 \in I(x).$$

The function $\Phi(t,x)$ is called the {\it evolution function} of the dynamical system: it associates
to every point in the set $M$ a unique image, depending on the variable $t$, called the {\it evolution parameter}.
$M$ is called {\it phase space} or {\it state space}, while the variable $x$ is called {\it initial state} of the system.

Varying parameters $T, M, \Phi$ one can define different kind of dynamical systems, for example, it is known the following cases:

{\it Real dynamical system:} A real dynamical system, {\it real-time dynamical system} or flow is a tuple $(T, M, \Phi)$ with $T$ an open interval in the real numbers $\mathbf R$, $M$ a manifold locally diffeomorphic to a Banach space, and $\Phi$ a continuous function.

{\it Cellular automaton:} A cellular automaton is a tuple $(T, M, \Phi)$, with $T$ the integers, $M$ a finite set, and $\Phi$ an evolution function.

{\it Discrete dynamical system:} A discrete dynamical system, {\it discrete-time dynamical system}, map or cascade is a tuple $(T, M, \Phi)$ with $T$ the integers, $M$ a manifold locally diffeomorphic to a Banach space, and $\Phi$ a function.

If $M=S^{m-1}=\left\{x=(x_1,\dots,x_m)\in \mathbf R^m: x_i\geq 0,\, \sum_{i=1}^m x_i=1\right\}$-simplex of probability measures on the set $E=\{1,2,\dots, m\}$
and $\Phi$ is a stochastic matrix, then the dynamical system describes a Markov chain \cite{sh}. Moreover if $M=S^{m-1}$ and
$\Phi$ is a quadratic stochastic operator \cite{gmr} then the dynamical system describes evolution of a population.

In idempotent mathematics the usual arithmetic operations are replaced with a new set of basic associative operations
(a new addition $\oplus$ and a new multiplication $\odot$) so that all the semifield or semiring axioms hold; moreover, the
new addition is idempotent, i.e., $x\oplus x = x$ for every element $x$ of the corresponding semiring, see, e.g., \cite{a,dd,dd1,l1,l2,m,z}.

A typical example is the semifield $\mathbf R_{\max}=\mathbf R\cup\{-\infty\}$ known as the Max-Plus algebra. This semifield consists of all real numbers and an additional element ${\bf 0}=-\infty$. This element ${\bf 0}$ is the zero element in $\mathbf R_{\max}$, and the basic operations are defined by the formulas $x\oplus y =\max\{x, y\}$ and $x\odot y = x + y$; the identity (or unit) element ${\bf 1}$ coincides with the usual zero 0.

In this paper, we consider $M$ as the simplex $\mathcal I_n$ of idempotent measures on $\{1,2,\dots,n\}$, where
 $$\mathcal I_n=\{(x_1,...,x_n)\in \mathbf R^{n}_{\rm max} : \max_{1\leq i\leq n}x_i=0\}=\{(x_1,...,x_n)\in{\mathbf R^{n}_{\rm max}}: x_1\oplus...\oplus x_n={\bf 1}\}$$ and consider linear maps of $\mathcal I_n$ to itself.
The term 'idempotent measure' in our case is used in the sense of
idempotent analysis (\cite{a,l1,l2,m,z}) rather than in the sense of abstract
harmonic analysis as in \cite{r}.

Idempotent measure theory has recently emerged as a
new branch of mathematics analysis for studying deterministic control problems and
first-order nonlinear partial differential equations such as Hamilton-Jacobi equations
with discontinuous initial data and low-lying eigenfunctions of Schr\"odinger operator.
During the past 20 years its applications have grown, establishing unexpected
connections with a number of other fields.

The linear dynamical systems of idempotent measures which we shall consider in
this paper are "idempotent" analogies of Markov chains, in our case a state
of the Markov chain is an idempotent measure, but linearity of the evolution
operator is defined by the usual operations $+$ and $\cdot$. In \cite{dd,dd1} an idempotent analogue of the Markov chain is introduced in their case the linearity of the evolution operator is defined by the new operations $\oplus$ and $\odot$.
Thus investigations of our paper has more relations with usual analysis than idempotent one.

The paper is organized as follows. In Section 2 we describe all linear operators which maps $n-1$-dimensional simplex of idempotent measures to itself. Such operators divided to two classes: the first class contains all $n\times n$-matrices with non-negative entries which has at least one zero-row; the second class contains all $n\times n$-matrices with non-negative entries which in each row and in each column has exactly one non-zero entry. These matrices play a role of the stochastic matrices in case of idempotent measures. Hence they can be called idempotent-stochastic matrices. Section 3 contains a description of fixed points for both classes of linear maps. The last Section is devoted to the dynamical systems generated by the linear maps of the set of idempotent measures.

\section{Linear maps of $\mathcal I_n$ to itself}

It is known (see, e.g., \cite{sh}) that the properties of homogeneous Markov chains with the phase
space $E = \{1,\dots,m\}$ can by completely determined by the initial distribution $x\in S^{m-1}$ and a
stochastic matrix $P$, i.e., $P$ maps $S^{m-1}$ to itself iff $P$ is a stochastic matrix. The theory of
dynamical systems (Markov chains) generated  by $P$ is known (see for example \cite{sh}).

In this section we shall answer the question: which linear map maps $\mathcal I_n$ to itself?
The following theorem describes such linear maps.

\begin{thm}\label{t1}
A linear operator $A=(a_{ij})_
{i,j=1,\dots, n}$ maps $\mathcal I_n$ to itself, if and only if it satisfies one of the following conditions:
\begin{itemize}
\item[(i)] $a_{ij}\geq0$ and $A$ has at least one zero-row;

\item[(ii)] $a_{ij}\geq0$ and each row and each column of $A$ contains exactly one non-zero element.
\end{itemize}
\end{thm}
\proof {\it Sufficiency}. 1. Assume the condition (i) is satisfied. For any $x\in \mathcal I_n$ we shall show that $y=A(x)\in \mathcal I_n$. If $i$th row of $A$ is zero row then  $y_i=0$.  Since $a_{kj}\geq 0$ and $x_j\leq 0$ we get
$$y_k=\sum_{j=1}^na_{kj}x_j\leq 0, \ \ \mbox{for all}\ \ k\ne i.$$ Hence
$y=A(x)\in \mathcal I_n$, for any $x\in \mathcal I_n.$

2. Assume now that the condition  (ii) is satisfied. Then there are
$i_1,i_2,...,i_n\in\{1,...,n\}$ such that $a_{ji_j}> 0$ and $a_{jk}=0$ if $k\ne i_j$. In this case the operator $y=A(x)$ has the following form
$$y_1=a_{1i_1}x_{i_1},\ \ y_2=a_{2i_2}x_{i_2},\,...\,, y_n=a_{ni_n}x_{i_n}.$$

By $x_i\leq0$ and $a_{ji_j}> 0$ we get $y_i\leq0$
for all $i=1,\dots,n$.  Since at least one of coordinates $x_j$ is equal to zero, at least one of $y_i=a_{ij}x_j$ is equal to zero. Hence $y=A(x)\in \mathcal I_n$.

{\it Necessariness.} Assume both conditions (i) and (ii) are not satisfied. The following three cases are possible:

 a) All rows of $A$ are non-zero and at least one row contains more than one non-zero element. Without lost of the generality we divide rows of $A$ in two class: in each row  $a^{(i)}=(a_{i1},...,a_{in})$,
$i=1,...,n_1$ there exists exactly one non-zero element $a_{ik_i}$; in each row
$a^{(i)}$, $i=n_1+1,...,n$ there are at least two non-zero elements. Then operator $A$ has the following form
$$\begin{array}{lllllll}
y_1=a_{1k_1}x_{k_1},\\
y_2=a_{2k_2}x_{k_2},\\
\dots\\
y_{n_1}=a_{n_1k_{n_1}}x_{k_{n_1}},\\
y_{n_1+1}=\sum_{j} a_{n_1+1j}x_j,\\
\dots\\
y_n=\sum_{j} a_{nj} x_j.
\end{array}$$

Take $j_0\in\{n_1+1,...,n\}$ and choose $x=(x_1,...,x_n)$ such that $x_{j_0}=0, \ x_j\neq0, \ j\neq j_0$. Then $y_i\neq 0, \
i=1,...,n_1$ and since in each $y_i, i=n_1+1,...,n$ there are at least two non-zero coefficient $a_{ij}$, we get $y_i\neq 0$,
for any $i=n_1+1,...,n$. Hence the condition $\max\{y_1,...,y_n\}=0$ is not satisfied and $y=A(x)\notin \mathcal I_n$.

b) All rows of $A$ are non-zero and at least one column contains more than one non-zero element, but each row has exactly one non-zero element. Then such a matrix contains at least one zero-column. Without lost of generality we assume that the last column is zero. Then coordinate $x_n$ of $x$ does not give any contribution to $A(x)$. If we take $x$ such that $x_n=0$ and $x_i<0$ for all $i\ne n$ then $y_i\ne 0$ for any $i=1,\dots,n$ so $y=A(x)\notin \mathcal I_n$.
  
c) Assume one of $a_{ij}$ is negative, let $a_{ik_i}<0$. We choose $x\in \mathcal I_n$ such that $x_{k_i}\neq 0$ and $x_j=0$ for all $j\neq k_i$. Then $a_{ik_i}x_{k_i}>0$, i.e. $i$-th coordinate of $y$ is positive. Hence
$y \notin \mathcal I_n$.
\endproof

The following proposition gives characteristic of the operator satisfying condition (ii).

Let $s_n$ be the group of permutations of ${1, 2, \dots, n}$.

\begin{pro}\label{p1} \begin{itemize}
\item[(a)] For any matrix $A$ satisfying condition (ii) there exists $\pi=\pi(A)\in s_n$ such that
$$A=A_\pi=\left(
  \begin{array}{cccc}
 a_{11} & \dots & a_{1n} \\[1.5mm]
 a_{21} & \dots & a_{2n} \\[1.5mm]
 \vdots &  \dots &\vdots \\[1.5mm]
a_{n1} &\dots & a_{nn} \\[1.5mm]
\end{array}
\right): \begin{array}{cc} a_{i\pi(i)}> 0, \ 1 \leq i
\leq n & \\[1.5mm] \mbox{and the rest of elements} \
a_{ij}=0 &\\[1.5mm]\end{array}.$$
\item[(b)] For any $\pi,\tau\in s_n$ the following equality holds
$$A_\pi A_\tau=A_{\tau\pi}.$$
The set $G=\{A_\pi: \pi\in s_n\}$ is a multiplicative group.
\end{itemize}
\end{pro}
\begin{proof}
(a) Since each row and each column of the matrix $A$
must contain exactly one non-zero element, there are distinct numbers
$i_1,i_2,...,i_n\in\{1,...,n\}$ such that $a_{ji_j}> 0$ and $a_{jk}=0$ if $k\ne i_j$. It is not difficult to
see that every such matrix $A$ corresponds to the permutation
$\pi=(\pi(j), j=1,\dots,n)$, with $\pi(j)=i_j$. This permutation is unique. Thus if given $A$ with
condition (ii) then we can construct $\pi=\pi(A)$, conversely, if given a permutation $\pi$ we can define matrix $A_\pi$
by $a_{ij}=0$, if $j\ne \pi(i)$ and $a_{ij}>0$, if $j=\pi(i)$.

(b) Take $A_\pi=\{a_{ij}\}$ and
$A_\tau=\{b_{ij}\}$. Let $A_\pi \cdot A_\tau=\{c_{ij}\}$. It is easy to see that
$$c_{ij}=\begin{cases}
0& \mbox{if} \ \ j\ne \tau(\pi(i));\\
a_{i\pi(i)}b_{\pi(i)\tau(\pi(i))}& \mbox{if} \ \
j=\tau(\pi(i)).\\
\end{cases}$$
This gives $A_\pi A_\tau=A_{\tau\pi}$
and then one easily can check that $G$ is a group.
\end{proof}

\section{Fixed points of $A$.}

Recall that fixed points of $A:\mathcal I_n\to \mathcal I_n$ are solutions to $A(x)=x$. Denote by Fix$(A)$ the set of all fixed points of $A$.

{\bf Case} (i): Let $A$ satisfies condition (i) of Theorem \ref{t1}. Then without lost of generality we assume that there is $n_0\in\{1,\dots,n\}$ such that all rows $a^{(j)}$, $j=n_0+1,\dots,n$ are zero-rows.
For such a linear operator $A$ the following lemma is obvious:

\begin{lemma}\label{l1} \begin{itemize}
\item[1.] The set $\mathcal I_{n,n_0}=\left\{x\in\mathcal I_n: x_j=0, \, j=n_0+1,...,n\right\}$ is invariant with respect to $A$, i.e. $A(\mathcal I_{n,n_0})\subset \mathcal I_{n,n_0}$. Moreover $A(x)\in \mathcal I_{n,n_0}$ for any $x\in \mathcal I_n$.

\item[2.] The set of fixed points {\rm Fix}$(A)$ is a subset of $\mathcal I_{n,n_0}$.
\end{itemize}
\end{lemma}
By this lemma the equation $A(x)=x$ has $x_j=0$ for all $j=n_0+1, \dots, n$ and can be reduced to the form
$A_{n_0}(x)=0$, where $x=(x_1,\dots,x_{n_0})$ and
$$A_{n_0}=\left(\begin{array}{cccc}
 a_{11}-1 & a_{12}&\dots & a_{1n_0} \\[1.5mm]
 a_{21} &a_{22}-1& \dots & a_{2n_0} \\[1.5mm]
 \vdots &  \dots &\vdots \\[1.5mm]
a_{n_01} &a_{n_02}&\dots & a_{n_0n_0}-1 \\[1.5mm]
\end{array}
\right).$$
If $\det(A_0)\ne 0$ then the equation $A_{n_0}(x)=0$ has unique solution
$x=(0,\dots,0)$. If  $\det(A_{n_0})=0$ and rank$(A_{n_0})=r$ then we can
assume that the first $r$ rows of $A_{n_0}$ are linearly independent,
consequently, the equation can be written as
\begin{equation}\label{n11}
x_i=-\sum_{j=r+1}^{n_0}d_{ij}x_j, \ \ i=1,\dots,r,
\end{equation}
where $d_{ij}={\det(A_{ij})\over\det(A_r)}$ with
$$A_{ij} =\left(\begin{array}{ccccccc}
a_{11}-1 &\dots & a_{1,i-1} & a_{1j} & a_{1,i+1} & \dots & a_{1r} \\[3mm]
a_{21} &\dots & a_{2,i-1} & a_{2j} & a_{2,i+1} & \dots & a_{2r} \\[3mm]
 &\dots &  & \dots &  & \dots & \\[3mm]
a_{r1} &\dots & a_{r,i-1} & a_{rj} & a_{r,i+1} & \dots & a_{rr}-1
\end{array}\right).
$$
An interesting problem is to find a necessary and sufficient
condition on matrix $D=(d_{ij})_{{i=1,\dots,r\atop
j=r+1,\dots,n_0}}$ under which the system (\ref{n11}) has unique
solution (remember that we are looking for solutions with $x_i\leq 0$).
The difficulty of the problem depends on rank $r$, here
we shall consider the case $r=n_0-1$.

\begin{pro}\label{pn} (cf. with Proposition 4.12 of \cite{clr})
\begin{itemize}
\item[1)] If $\det(A_{n_0})\ne 0$ then
the operator $A$ has unique fixed point
$(0,...,0)\in \mathcal I_n$.

\item[2)] If $\det(A_{n_0})=0$ and rank$(A_{n_0})=n_0-1$ then operator $A$ has unique fixed point $(0,...,0)$ if and only if
\begin{equation}\label{con}
\det(A_{i_0n_0})\cdot\det(A_{n_0-1})>0,
\end{equation}
for some $i_0\in \{1,\dots,n_0-1\}$.

\item[3)] If the condition (\ref{con}) is not satisfied then operator $A$ has infinitely many fixed points $x=x(\alpha)$ of the form
$$x(\alpha)=\left(-{\det(A_{1n_0})\over \det(A_{n_0-1})}\alpha,\, -{\det(A_{2n_0})\over \det(A_{n_0-1})}\alpha,\, \dots, \,
-{\det(A_{n_0-1n_0})\over \det(A_{n_0-1})}\alpha,\,\alpha,\, 0,0,\dots,0\right),$$ where $\alpha\leq 0.$
\end{itemize}
\end{pro}
\proof 1) Straightforward.

2) If rank$(A_{n_0})=n_0-1$ then from (\ref{n11}) we get
\begin{equation}\label{n12}
x_i=-{\det(A_{in_0})\over \det(A_{n_0-1})}x_{n_0}, \ \ i=1,\dots,n_0-1.
\end{equation}
From (\ref{n12}) it follows that the condition (\ref{con}) is
necessary and sufficient to have unique solution $(0,\dots,0)$.

3) Follows from (\ref{n12}) taking $\alpha=x_{n_0}$.
\endproof
Now we shall give full description of fixed points of $A$ for $n=3$ and $n_0=2$.
\begin{pro}\label{p2}
 If $n=3$ and $n_0=2$ then {\rm Fix}$(A)$ has the following form
$${\rm Fix}(A)=\left\{\begin{array}{lllll}
\{(c\alpha,\alpha,0), \alpha\leq 0\}, \ \ \mbox{if} \ \ \det(A_2)=0, \,
a_{11}<1;\\
\{(\alpha,\beta,0),\ \ \alpha\leq 0,\ \ \beta\leq 0\}, \ \ \mbox {if}\
\ a_{12}=a_{21}=0,\ \ a_{11}=a_{22}=1;\\
\{(0,0,0)\}\ \ \mbox{otherwise}
\end{array}\right.$$
where $c=\frac{a_{12}}{1-a_{11}}$.
\end{pro}
\proof The proof consists a detailed analysis of equation $A_2(x)=x$.
\endproof

 {\bf Case} (ii): Assume $A$ satisfies the condition (ii) of Theorem \ref{t1}. In this case by Proposition \ref{p1} $A$ is represented as $A_\pi$.

 Let us first consider an example:

{\bf Example 1.} Consider matrix
 $$A=\left(\begin{array}{ccccccc}
0 &\alpha &0 &0 &0 \\[2mm]
\beta &0 &0 &0 &0 \\[2mm]
0 &0 &0 &0 &\gamma \\[2mm]
0 &0 &0 &\delta &0 \\[2mm]
0 &0 &\eta &0 &0 \\[2mm]
\end{array}\right), \ \ \alpha\beta\gamma\delta\eta>0.
$$
Corresponding $\pi$ is $\pi=(12)(35)(4)$ and equation $A(x)=x$ has the following form:
\begin{equation}\label{e} x_1=\alpha x_2,\ \ x_2=\beta x_1,\ \ x_3=\gamma x_5, \ \ x_4=\delta x_4, \ \ x_5=\eta x_3.
\end{equation}
 From system (\ref{e}) we get
 \begin{equation}\label{e1} x_1=\alpha\beta x_1,\ \ x_3=\gamma\eta x_3, \ \ x_4=\delta x_4.
\end{equation}
Consequently we have
$${\rm Fix}(A)=\left\{\begin{array}{llllllll}
\{(0,0,0,0,0)\}, \ \ \mbox{if} \ \ \alpha\beta\ne 1,\, \gamma\eta\ne 1, \, \delta\ne 1;\\[2mm]
\{(x_1,\beta x_1,0,0,0), \, x_1\leq 0\}, \ \ \mbox{if} \ \ \alpha\beta=1,\, \gamma\eta\ne 1, \, \delta\ne 1;\\[2mm]
\{(0,0,x_3,0,\eta x_3), \, x_3\leq 0\}, \ \ \mbox{if} \ \ \alpha\beta\ne 1,\, \gamma\eta=1, \, \delta\ne 1;\\[2mm]
\{(0,0,0,x_4,0), \, x_4\leq 0\}, \ \ \mbox{if} \ \ \alpha\beta\ne 1,\, \gamma\eta\ne 1, \, \delta=1;\\[2mm]
\{(x_1,\beta x_1,x_3,0,\eta x_3), \, x_1\leq 0, x_3\leq 0\}, \ \ \mbox{if} \ \ \alpha\beta= 1,\, \gamma\eta= 1, \, \delta\ne 1;\\[2mm]
\{(x_1,\beta x_1,0,x_4,0), \, x_1\leq 0, x_4\leq 0\}, \ \ \mbox{if} \ \ \alpha\beta= 1,\, \gamma\eta\ne 1, \, \delta=1;\\[2mm]
\{(0,0,x_3,x_4,\eta x_3), \, x_3\leq 0, x_4\leq 0\}, \ \ \mbox{if} \ \ \alpha\beta\ne 1,\, \gamma\eta= 1, \, \delta=1;\\[2mm]
\{(x_1,\beta x_1,x_3,x_4,\eta x_3), \, x_1\leq 0, x_3\leq 0, x_4\leq 0, x_1x_3x_4=0\}, \ \ \mbox{if} \ \ \alpha\beta= \gamma\eta=\delta=1.\\[2mm]
\end{array}\right.$$

The following theorem generalizes Example 1 for arbitrary $A$ satisfying condition (ii).

\begin{thm}\label{t2} If permutation $\pi$  has the following decomposition into disjoint cycles
$$\pi=(i_{11},i_{12},...,i_{1k_1})(i_{21},i_{22},...,i_{2k_2})...(i_{q_1},i_{q_2},...,i_{qk_q}),\, 1\leq q\leq n, \, 1\leq k_j\leq n,$$
then
\begin{itemize}
\item[(1)] If $a_{i_{j1}i_{j2}}a_{i_{j2}i_{j3}}...a_{i_{jk_j}i_{j1}}\neq 1$, for all $j=1,\dots,q$ then the fixed point $(0,0,...,0)$ of $A_\pi$ is unique.

\item[(2)] If there are $j_1,j_2,...,j_p$, with some $p\in \{1,\dots,q\}$  such that
$a_{i_{j_s1}i_{j_s2}}...a_{i_{j_sk_{j_s}}i_{j_s1}}=1$, $s=1,\dots, p$ then
there are infinitely many fixed points with $p$ free coordinates if $p<q$ and with $q-1$ free coordinates if $p=q$.
\end{itemize}
\end{thm}
\proof The equation $Ax=x$ has the following form
\begin{equation}\label{e2}
x_{i_{js}}=a_{i_{js}i_{js+1}}x_{i_{js+1}}\ \ s=1,\dots,k_j, \ \ j=1,\dots,q.
\end{equation}
From (\ref{e2}) we get
\begin{equation}\label{e3}
a_{i_{j1}i_{j2}}a_{i_{j2}i_{j3}}...a_{i_{jk_j}i_{j1}}x_{i_{j1}}=x_{i_{j1}}.
\end{equation}
If $a_{i_{j1}i_{j2}}a_{i_{j2}i_{j3}}...a_{i_{jk_j}i_{j1}}\neq 1$ then equation (\ref{e3}) has unique solution
$x_{i_{j1}}=0$ and if $a_{i_{j1}i_{j2}}a_{i_{j2}i_{j3}}...a_{i_{jk_j}i_{j1}}=1$ then it has infinitely many solutions $x_{i_{j1}}\leq 0$. This completes the proof.
\endproof

\section{Behavior of trajectories}

Let $x^{(0)}\in \mathcal I_n$ be an initial point, and let $\{x^{(0)}, x^{(1)}, x^{(2)},\dots\}$ be the trajectory (dynamical system) of the point $x^{(0)}$ with respect to operator $A$, i.e. $x^{(m+1)}=A(x^{(m)}), m=0,1,\dots$.
Denote by $\omega(x^{(0)})$ the set of limit points of the trajectory $x^{(m)}, m\geq 0$. Since $\{x^{(m)}\}_{m=0}^\infty\subset\mathcal I_n$ and $\mathcal I_n$
is a compact set \cite{z}, it follows that $\omega(x^{(0)})\ne \emptyset$. If $\omega(x^{(0)})$ consists of a single point, then the trajectory converges, and the point is a fixed point of the operator $A$.

Let us first investigate the trajectory $x^{(m)}, m\geq 0$ when the initial point $x^{(0)}$ contains a coordinate equal to $-\infty$. For given operator $A$ satisfying (i) or (ii) we define a directed pseudograph (or multigraph i.e. a "graph" which can have both multiple edges and loops), $G_A=(V,L)$ with vertices $V=\{1,2,\dots,n\}$ and edges $L$ as
$$L=\{\langle i,j\rangle: \ \ {\rm with\, direction \, from}\ \ i \ \ {\rm to} \ \ j \ \ {\rm if} \ \  a_{ji}>0\}.$$ Note that this graph contains a loop $\langle i,i\rangle$ if $a_{ii}>0$, the graph contains two edges connecting $i$ and $j$  if $a_{ij}a_{ji}>0$. In the last case one edge directed from $i$ to $j$, another one is directed from $j$ to $i$.

\begin{pro}\label{pg} \begin{itemize}

\item[a)] If $x^{(0)}$ has not any coordinate equal to $-\infty$ then $x^{(m)}$ does not contain any coordinate equal to $-\infty$ for any finite $m\geq 1$.

\item[b)] If graph $G_A$ does not contain any cycle (taking directions into account) and $x^{(0)}$ has some coordinates equal to $-\infty$ then the vectors $x^{(m)}$, $m\geq n$ do not contain any coordinate equal to $-\infty$.

\item[c)] If the collection of vertices $\{i_1,\dots,i_q\}\subset V$, $q\geq 1$ gives a cycle on the graph $G_A$ and if $x_{i_k}^{(0)}=-\infty$ for some $k=1,\dots,q$ then the vectors $x^{(m)}$ have at least one coordinate equal to $-\infty$  for any $m\geq 1$. Moreover, the vectors $u_j=(u_{j,1},\dots, u_{j,n})\in \mathcal I_n$, with $u_{j,i}=0$ if $i\ne i_j$ and $u_{j,i}=-\infty$ if $i=i_j$, $j=1,\dots,q$ form a $q$-cycle with respect to the operator $A$.
     \end{itemize}
\end{pro}
\proof a) Straightforward.

b)  Since the graph does not contain any cycle, it only contains (directed) paths. Assume the longest directed path of the graph $G_A$ is $p=p_1,p_2,\dots,p_k=q$, $k\leq n$, from $p\in V$ to $q\in V$, $p\ne q$ and suppose $x_{p_1}^{(0)}=-\infty$ then since $a_{p_2p_1}>0$ we have $x_{p_2}^{(1)}=-\infty$. Similarly we get $x_{p_i}^{(i-1)}=-\infty$ for $i=3,\dots,k$. Since $a_{jp_k}=0$ for any $j=1,\dots,n$ (otherwise the path is not the longest) the last $-\infty$ does not give any contribution to $x^{(k)}$.
This argument also works for any shorter path, i.e. after finite steps (the length of the path) a $"-\infty"$ coordinate
does not give any contribution to the next terms of the trajectory. Since the length of the longest path can not be larger than $n$, starting from step $n$ all $"-\infty"$ coordinates disappear from $x^{(m)}$, $m\geq n$.

c) Without lost of generality we assume that $i_j=j$ for all $j=1,\dots,q$ and assume that $x_1^{(0)}=-\infty$. Then it is easy to see that $x_2^{(1)}=-\infty$, $x_3^{(2)}=-\infty$,$\dots$. Since the collection of vertices $\{1,\dots,q\}$ is a cycle, the $-\infty$ coordinate "travels" cyclicly as follows:
 \begin{equation}\label{cp}
 x_1^{(qs)}\to x_2^{(qs+1)}\to\dots\to x_{q}^{(qs+q-1)}\to x_1^{(q(s+1))}, s\geq 0.
  \end{equation}
  Hence each vector $x^{(m)}$ has at least one coordinate equal to $-\infty$ for any $m\geq 1$.
  If $x^{(0)}=u_j$ for some $j$ then the cycle (\ref{cp}) gives the $q$-cycle $u_1\to u_2\to\dots\to u_q\to u_1$.
\endproof

{\bf Case} (i): Using notations of the previous section we get that $x^{(m)}_j=0$ for any $j=n_0+1,\dots,n$ and $m\geq 1$. Moreover
$x^{(m)}\in \mathcal I_{n,n_0}$, $m\geq 1$. Thus the investigation of trajectory $x^{(m)}, m\geq 0$ can be reduced to investigation of trajectory $u^{(m+1)}=B(u^{(m)}), \, m\geq 0$ with operator $B: \mathcal J_{n_0}\to\mathcal J_{n_0}$ where
   $$\mathcal J_n=\{x\in \mathbf R_{\rm max}^n: x_i\leq 0, \forall i=1,\dots,n\},$$
and $B=(a_{ij})_{i,j=1,\dots,n_0}$ is the $n_0\times n_0$-minor of matrix $A$.

   It is known (see, for example, \cite{d}) that the eigenvalues and eigenvectors of $B$ determine the behavior of trajectories. For example, if $u$ is an eigenvector of $B$, with a real eigenvalue smaller than one, then the straight lines given by the points along $\alpha u$, with $\alpha\in \mathbf R$, is an invariant curve of the map $B$. Points in this straight line run into the fixed point.

   The following theorem is a corollary of a known Theorem of the theory of linear dynamical systems \cite{d}.

   \begin{thm} Suppose $B:\mathcal J_{n_0}\to \mathcal J_{n_0}$ has eigenvalues $\lambda_1,\dots,\lambda_{n_0}$
   \begin{itemize}
   \item[1.] If $|\lambda_i|<1$, $i=1,\dots,n_0$ then $\omega(x^{(0)})=\{(0,\dots,0)\}$ for any $x^{(0)}\in \mathcal I_n$.
   \item[2.] If $|\lambda_i|>1$, $i=1,\dots,n_0$ then $\lim_{m\to +\infty}x^{(-m)}=(0,\dots,0)$ for any $x^{(0)}\in \mathcal I_n$.
   \item[3.] If $|\lambda_i|<1$, $i=1,\dots,p$ and $|\lambda_i|>1$, $i=p+1,\dots,n_0$ then
   there is a $p$ dimensional space $W^s\subset \mathbf R^n_{max}$ and a $n_0-p$ dimensional space $W^u\subset \mathbf R^n_{max}$ on which
   \begin{itemize}
   \item[3.a)]
   If $x^{(0)}\in W^s\cap \mathcal I_n$, then $\omega(x^{(0)})=\{(0,\dots,0)\}$.
    \item[3.b)]
   If $x^{(0)}\in W^u\cap \mathcal I_n$, then $\lim_{m\to +\infty}x^{(-m)}=(0,\dots,0)$.
   \item[3.c)] If $x\in \mathcal I_n\setminus (W^s\cup W^u)$ then $\lim_{m\to +\infty}|x^{(m)}|=+\infty$.
   \end{itemize}
   \end{itemize}
     \end{thm}
Let us now consider an example:

{\bf Example 2.} Consider
$A=\left(\begin{array}{cc}
a_{11}& a_{12} \\
0& 0
\end{array}\right).$
In this case we have $A^n=\left(\begin{array}{cc}
a^n_{11}& a^{n-1}_{11}a_{12} \\
0& 0
\end{array}\right)$. Consequently, for any $x^{(0)}=(x^{(0)}_1, x^{(0)}_2)\in \mathcal I_2$ we obtain
$$x^{(n)}=A^n(x^{(0)})=(a^n_{11}x^{(0)}_1+a^{n-1}_{11}a_{12}x^{(0)}_2,0)$$ and
$$\lim_{n\rightarrow\infty}x^{(n)}=\lim_{n\rightarrow\infty}(a^n_{11}x^{(0)}_1+a^{n-1}_{11}a_{12}x^{(0)}_2,0)=
\left\{\begin{array}{lll}(0,0),\ \ \ \ \ \ \ \ \ \ \ \ \ \ \ \mbox{if}\ a_{11}<1,\\
(x^{(0)}_1+a_{12}x^{(0)}_2,0),\ \mbox{if}\ a_{11}=1,\\
(-\infty,0),\ \ \ \ \ \ \ \ \ \ \ \mbox{if}\
a_{11}>1.
\end{array}\right.$$

{\bf Case} (ii): In this subsection we shall study trajectory $x^{(m)}, m\geq 0$ for a
linear operator $A$ which satisfies condition (ii) of Theorem \ref{t1}.

 Suppose the operator $A=A_\pi$ corresponds to the permutation $\pi$  which has the following decomposition into disjoint cycles
$$\pi=(i_{11},i_{12},...,i_{1k_1})(i_{21},i_{22},...,i_{2k_2})...(i_{q_1},i_{q_2},...,i_{qk_q}),\, 1\leq q\leq n, \, 1\leq k_j\leq n.$$
By Proposition \ref{p1} we have $A^m_\pi=A_{\pi^m}$, where $\pi^m=\pi^{m-1}\pi$ is $m$-th iteration of $\pi$.

Denote $Q_p=a_{i_{p1}i_{p2}}a_{i_{p2}i_{p3}}\dots a_{i_{p(k_p-1)}i_{pk_p}}a_{i_{pk_p}i_{p1}}$,\ \ $p=1,\dots,q$.
The following theorem gives complete description of the trajectory:

\begin{thm}\label{t5} If $i\in \{i_{p1},i_{p2},...,i_{pk_p}\}$ for some $p=1,\dots,q$ then

\begin{equation}\label{u}\lim_{s\to\infty}x_i^{(k_ps+r)}=\left\{\begin{array}{llll}
0, \ \ \mbox{if} \ \ Q_p<1;\\[2mm]
\left(\prod_{j=0}^{r-1} a_{\pi^j(i)\pi^{j+1}(i)}\right)x^{(0)}_{\pi^{r}(i)},\ \ \mbox{if} \ \ Q_p=1, \, 0\leq r<k_p;\\[2mm]
-\infty,\ \ \mbox{if} \ \ Q_p>1,
\end{array}\right.
\end{equation}
\end{thm}
\proof From $x^{(m)}=A_\pi(x^{(m-1)})$ we get
\begin{equation}\label{i}
x_i^{(m)}=a_{i\pi(i)}x^{(m-1)}_{\pi(i)}=
a_{i\pi(i)}a_{\pi(i)\pi^2(i)}x^{(m-2)}_{\pi^2(i)}=\dots
=\prod_{j=0}^{m-1}a_{\pi^j(i)\pi^{j+1}(i)}x^{(0)}_{\pi^m(i)}.
\end{equation}
For $i\in \{i_{p1},i_{p2},...,i_{pk_p}\}$ and $m=k_ps+r$ from  (\ref{i}) we get
\begin{equation}\label{ii}
x_i^{(k_ps+r)}=Q_p^s\prod_{j=0}^{r-1}a_{\pi^j(i)\pi^{j+1}(i)}x^{(0)}_{\pi^r(i)}.
\end{equation}
This equality completes the proof.
\endproof
\begin{rk} By Theorem \ref{t2} we know that in case $Q_p=1$ there are infinitely many fixed points.
The equality (\ref{u}) shows that in this case the limit of the $i$th coordinate does not exist,
it depends on the remainder $r$.
\end{rk}

{\bf Example 3.} Consider $A=\left(\begin{array}{cc}
 0& a_{12} \\
a_{21}& 0
\end{array}\right)$. It is easy to get
$$A^{2s}=\left(\begin{array}{cc}
a^s_{12}a^s_{21}& 0 \\
0& a^s_{21}a^s_{12}
\end{array}\right) \ \ {\rm and}\ \
A^{2s+1}=\left(\begin{array}{cc}
0& a^{s+1}_{12}a^{s}_{21} \\[2mm]
a^{s}_{12}a^{s+1}_{21}& 0
\end{array}\right),\, s\geq 0.$$

Consequently, $$x^{(m)}=\left\{\begin{array}{ll}
((a_{12}a_{21})^sx^{(0)}_1,\, (a_{12}a_{21})^sx^{(0)}_2),  \ \ \mbox{if} \ \
m=2s,\\[2mm]
((a_{12}a_{21})^sa_{12}x^{(0)}_2,\, (a_{12}a_{21})^sa_{21}x^{(0)}_1)\ \ \mbox{if} \ \ m=2s+1.
\end{array}\right.$$
Hence

$$\lim_{s\rightarrow\infty}x^{(2s+r)}=\left\{\begin{array}{llll}
(0,0),\ \ \ \ \ \mbox{if}\ a_{12}a_{21}<1,\\[2mm]
(x_1^{(0)},x_2^{(0)}), \ \ \mbox{if}\ a_{12}a_{22}=1, r=0,\\[2mm]
(a_{12}x^{(0)}_2,a_{21}x^{(0)}_1),\ \ \mbox{if}\ a_{12}a_{22}=1, r=1,\\[2mm]
(-\infty,0)\, {\rm or} \,(0,-\infty),\ \ \mbox{if}\ a_{12}a_{21}>1.\\
\end{array}\right.$$

\section*{ Acknowledgements}

 U. Rozikov thanks Institut des Hautes \'Etudes Scientifiques (IHES), Bures-sur-Yvette, France for support of his visit to IHES and IMU/CDC-program for a (travel) support.

{}
\end{document}